\documentclass[a4paper]{amsart}
\usepackage{amssymb}

\input xy
\xyoption{all}

\newcommand{\Id}{\mathrm{Id}}
\newcommand{\Comp}{\mathsf{Comp}}

\newcommand{\A}{\mathcal A}

\newcommand{\C}{\mathcal C}

\newcommand{\B}{\mathcal B}

\newcommand{\N}{\mathbb N}
\newcommand{\R}{\mathbb R}
\newcommand{\T}{\mathbb T}
\newcommand{\F}{\mathcal F}
\newcommand{\M}{\mathbb M}
\newcommand{\V}{\mathbb V}

\newtheorem{theorem}{Theorem}

\newtheorem{lemma}{Lemma}
\newtheorem{remark}{Remark}

\newtheorem{proposition}{Proposition}

\begin{document}

\title{A functional representation of the capacity multiplication monad}
\author{Taras Radul}

\maketitle

Institute of Mathematics, Casimirus the Great University of Bydgoszcz, Poland;
\newline
Department of Mechanics and Mathematics, Ivan Franko National University of Lviv,
Universytettska st., 1. 79000 Lviv, Ukraine.
\newline
e-mail: tarasradul@yahoo.co.uk

\textbf{Key words and phrases:}  Monad, capacity, fuzzy integral, triangular norm.

\subjclass[MSC 2010]{18B30, 18C15, 28E10, 54B30}

\begin{abstract}  Functional representations of the capacity monad based on the max and min operations were considered in \cite{Ra1} and  \cite{Ny1}. Nykyforchyn considered in \cite{Ny2} some alternative monad structure for the possibility capacity functor based on the max and usual multiplication operations.
We show that such capacity monad (which we call the capacity multiplication monad) has a functional representation, i.e. the space of capacities on a compactum $X$ can be naturally embedded (with preserving of the monad structure) in some space of functionals on $C(X,I)$. We also describe this space of functionals  in terms of properties of functionals.
\end{abstract}

\maketitle

\section{Introduction}
Functional representations of monads (i.e. natural embeddings into $\R^{C(X,S)}$ which preserves a monad structure where $S$ is a subset of $\R$) were considered in \cite{Ra2} and  \cite{Ra3}. Some functional representations of hyperspace monad were constructed in \cite{Ra4} and  \cite{Ra5}.

Capacities (non-additive measures, fuzzy measures) were introduced by Choquet in \cite{Ch} as a natural generalization of additive measures. They found numerous applications (see for example \cite{EK},\cite{Gil},\cite{Sch}). Categorical and topological properties of spaces of upper-semicontinuous capacities on compact Hausdorff spaces were investigated in \cite{NZ}. In particular, there was built the capacity functor which is a functorial part of a capacity monad $\M$ based on the max and min operations.

Well known is the Choquet integral, which is, in fact, some functional representation of the functor $M$, i.e., the space of capacities $MX$ can be naturally embedded in $\R^{C(X)}$. But this representation does not preserve the monad structure.
Nykyforchyn using the Sugeno integral  provided a functional representation of capacities as functionals on the space  $C(X,I)$ which preserves the monad structure \cite{Ny1}. Some modification of the Sugeno integral yields a functional representation of capacities as functionals on the space  $C(X)$ \cite{Ra1}.

Let us remark that the min operation is a triangular norm on the unit interval $I$. Another important triangular norm is the multiplication operation. Nykyforchyn build in  \cite{Ny2} a capacity monad based on the max and  multiplication operations. (Let us remark that recently Zarichnyi proposed to use triangular norms to construct monads \cite{Za}). The main aim of this paper is to find a  representation of the monad from \cite{Ny2}. We use a fuzzy integral based on the max and  multiplication operations for this purpose.

\section{Capacities and monads}
 By $\Comp$ we denote the category of compact Hausdorff
spaces (compacta) and continuous maps. For each compactum $X$ we denote by $C(X)$ the Banach space of all
continuous functions $\phi:X\to\R$ with the usual $\sup$-norm: $
\|\phi\| =\sup\{|\phi(x)|\mid x\in X\}$. We also consider on $C(X)$ the natural partial order.

In what follows, all spaces and maps are assumed to be in $\Comp$ except for $\R$,
the spaces $C(X)$ and functionals defined on $C(X)$ with $X$ compact Hausdorff.

We recall some categorical notions (see \cite{RZ} and \cite{TZ}
for more details). We define them only for the category $\Comp$. The central notion is the notion of monad (or triple) in the sense of S.Eilenberg and J.Moore.

A {\it monad} \cite{EM} $\T=(T,\eta,\mu)$ in the category
$\Comp$ consists of an endofunctor $T:{\Comp}\to{\Comp}$ and
natural transformations $\eta:\Id_{\Comp}\to T$ (unity),
$\mu:T^2\to T$ (multiplication) satisfying the relations $\mu\circ
T\eta=\mu\circ\eta T=${\bf 1}$_T$ and $\mu\circ\mu T=\mu\circ
T\mu$. (By $\Id_{\Comp}$ we denote the identity functor on the
category ${\Comp}$ and $T^2$ is the superposition $T\circ T$ of
$T$.)

Let $\T=(T,\eta,\mu)$ be a monad in the category ${\Comp}$. The
pair $(X,\xi)$ where $\xi:TX\to X$ is a map is called a $\T$-{\it
algebra} if $\xi\circ\eta X=id_X$ and $\xi\circ\mu X=\xi\circ
T\xi$. Let $(X,\xi)$, $(Y,\xi')$ be two $\T$-algebras. A map
$f:X\to Y$ is called a $\T$-algebras morphism if $\xi'\circ
Tf=f\circ\xi$.

A natural transformation $\psi:T\to T'$ is called a {\it morphism}
from a monad $\T=(T,\eta,\mu)$ into a monad $\T'=(T',\eta',\mu')$
if $\psi\circ\eta= \eta'$ and $\psi\circ\mu=\mu'\circ\eta T'\circ
T\psi$. If all of the components of $\psi$ are monomorphisms then
the monad $\T$ is called a {\it submonad} of $\T'$ and $\psi$ is
called a {\it monad embedding}.

 Let $A$ be a subset of $X$.  By $\F(X)$ we denote the family of all closed subsets of $X$. Put $I=[0,1]$.

 We follow a terminology from \cite{NZ}.
A function $\nu:\F(X)\to I$  is called an {\it upper-semicontinuous capacity} on $X$ if the three following properties hold for each closed subsets $F$ and $G$ of $X$:

1. $\nu(X)=1$, $\nu(\emptyset)=0$,

2. if $F\subset G$, then $\nu(F)\le \nu(G)$,

3. if $\nu(F)<a$, then there exists an open set $O\supset F$ such that $\nu(B)<a$ for each compactum $B\subset O$.

A capacity $\nu$ is extended in \cite{NZ} to all open subsets $U\subset X$ by the formula $\nu(U)=\sup\{\nu(K)\mid K$ is a closed subset of $X$ such that $K\subset U\}$.

It was proved in \cite{NZ} that the space $MX$ of all upper-semicontinuous  capacities on a compactum $X$ is a compactum as well, if a topology on $MX$ is defined by a subbase that consists of all sets of the form $O_-(F,a)=\{c\in MX\mid c(F)<a\}$, where $F$ is a closed subset of $X$, $a\in [0,1]$, and $O_+(U,a)=\{c\in MX\mid c(U)>a\}$, where $U$ is an open subset of $X$, $a\in [0,1]$. Since all capacities we consider here are upper-semicontinuous, in the following we call elements of $MX$ simply capacities.

A capacity $\nu\in MX$ for a compactum $X$ is called  a necessity (possibility) capacity if for each family $\{A_t\}_{t\in T}$ of closed subsets of $X$ (such that $\bigcup_{t\in T}A_t$ is a closed subset of $X$) we have $\nu(\bigcap_{t\in T}A_t)=\inf_{t\in T}\nu(A_t)$  ($\nu(\bigcup_{t\in T}A_t)=\sup_{t\in T}\nu(A_t)$). (See \cite{WK} for more details.) We denote by $M_\cap X$ ($M_\cup X$) a subspace of $MX$ consisting of all necessity (possibility) capacities. Since $X$ is compact and $\nu$ is upper-semicontinuous, $\nu\in M_\cap X$ iff $\nu$ satisfy the simpler requirement that $\nu(A\cap B)=\min\{\nu(A),\nu(B)\}$. 

If $\nu$ is a capacity on a compactum $X$, then  the function $\kappa X(\nu)$, that is defined on the family $\F(X)$  by the formula $\kappa X(\nu)(F) = 1-\nu (X\setminus F)$, is a capacity as well. It is called the dual
capacity (or conjugate capacity ) to $\nu$. The mapping $\kappa X : MX \to MX$ is a homeomorphism and an involution \cite{NZ}. Moreover, $\nu$ is a necessity capacity if and only if $\kappa X(\nu)$ is a possibility capacity. This implies in particular that $\nu\in M_\cup X$ iff $\nu$ satisfy the simpler requirement that $\nu(A\cup B)=\max\{\nu(A),\nu(B)\}$. It is easy to check that $M_\cap X$ and $M_\cup X$ are closed subsets of $MX$.


The assignment $M$ extends to the capacity functor $M$ in the category of compacta, if the map $Mf:MX\to MY$ for a continuous map of compacta $f:X \to Y$ is defined by the formula $Mf(c)(F)=c(f^{-1}(F))$ where $c\in MX$ and $F$ is a closed subset of $X$. This functor was completed to the monad $\M=(M,\eta,\mu)$ \cite{NZ}, where the components of the  natural transformations are defined as follows: $\eta X(x)(F)=1$ if $x\in F$ and $\eta X(x)(F)=0$ if $x\notin F$;
$\mu X(\C)(F)=\sup\{t\in[0,1]\mid \C(\{c\in MX\mid c(F)\ge t\})\ge t\}$, where $x\in X$, $F$ is a closed subset of $X$ and $\C\in M^2(X)$ (see \cite{NZ} for more details).

It was shown in \cite{HN} that $M_\cup$ and $M_\cap$ are subfunctors of $M$ and if we take corresponding restrictions of the functions $\mu X$, we obtain submonads $\M_\cup$ and $\M_\cap$ of the monad $\M$.

The semicontinuity of capacities yields that we can change $\sup$ for  $\max$ in the definition of the map $\mu X$. More precisely, existing of $\max$ follows from Lemma 3.7 \cite{NZ}. For a closed set $F\subset X$ and for $t\in I$ put $F_t=\{c\in MX\mid c(F)\ge t\}$. We can rewrite the definition of the map $\mu X$ as follows $\mu X(\C)(F)=\max\{\C(F_t)\wedge t\mid t\in(0,1]\}$.

Let us remark that the operation $\wedge$ is a triangular norm. It seems naturally to consider instead $\wedge$ another triangular norm. Define the map $\mu^\bullet X:M^2 X\to MX$  by the formula $\mu^\bullet X(\C)(F)=\max\{\C(F_t)\cdot t\mid t\in(0,1]\}$. (Existing of $\max$ as well follows from Lemma 3.7 \cite{NZ}.)

\begin{proposition}\label{CE} The natural transformation $\mu^\bullet$ does not satisfy the property $\mu^\bullet\circ\mu^\bullet M=\mu^\bullet\circ M\mu^\bullet$.
\end{proposition}

\begin{proof} Consider $X=\{a,b\}$, where $\{a,b\}$ is a two-point discrete space.  Define $\A_1\in M^2 X$ as follows $\A_1(\alpha)=1$ iff $\alpha\supset\{a\}_\frac{1}{2}$ and $\A_1(\alpha)=0$ otherwise for $\alpha\in\F(MX)$. Define $\A_2\in M^2 X$ as follows $\A_2(\alpha)=1$ iff $\alpha=MX$, $\A_2(\alpha)=\frac{1}{2}$ iff  $\alpha\supset\{a\}_1$ and $\A_1(\alpha)=0$ otherwise for $\alpha\in\F(MX)$. Now, define $\gimel\in M^3(X)$ by the formula  $\gimel(\Lambda)=\frac{1}{2}\eta M^2X(\A_1)(\Lambda)+\frac{1}{2}\eta M^2X(\A_2)(\Lambda)$ for $\Lambda\in\F(M^2X)$.

We have $\mu^\bullet X\circ M(\mu^\bullet X)(\gimel)(\{a\})=\max\{\gimel((\mu^\bullet X)^{-1}(\{a\}_t))\cdot t\mid t\in(0,1]\}$. It is easy to see that  $\mu^\bullet X(\A_1)=\mu^\bullet X(\A_2)=\frac{1}{2}$. Then $\gimel((\mu^\bullet X)^{-1}(\{a\}_\frac{1}{2}))\cdot \frac{1}{2}=1\cdot \frac{1}{2}= \frac{1}{2}$. Hence we obtain $\mu^\bullet X\circ\mu^\bullet MX(\gimel)(\{a\})\ge\frac{1}{2}$.

On the other hand $\mu^\bullet X\circ\mu^\bullet MX(\gimel)(\{a\})=\max\{\mu^\bullet MX(\gimel)(\{a\}_t))\cdot t\mid t\in(0,1]\}=\max\{\max\{\gimel((\{a\}_t)_s)\cdot s\mid s\in(0,1]\}\cdot t\mid t\in(0,1]\}$. The function $\delta(s,t)=\gimel((\{a\}_t)_s)$ is nonincreasing on both variables. We have $\delta(s,t)=0$ for each $(s,t)$ such that $s>\frac{1}{2}$ and $t>\frac{1}{2}$. Moreover  $\delta(1,\frac{1}{2})=\delta(\frac{1}{2},1)=\frac{1}{2}$. Hence $\mu^\bullet X\circ\mu^\bullet MX(\gimel)(\{a\})=\max\{\max\{\gimel((\{a\}_t)_s)\cdot s\mid s\in(0,1]\}\cdot t\mid t\in(0,1]\}=\frac{1}{4}$.
\end{proof}

\begin{remark}Since the triple  $\M^\bullet=(M,\eta,\mu^\bullet)$ does not form a monad, the problem of uniqueness of the monad $\M$ stated in  \cite{NZ} is still open.
\end{remark}

But things may turn out differently if we restrict the map $\mu^\bullet X$ to the set $M_\cup(M_\cup X)\subset M(MX)$. It is easy to see that for such restriction  we can consider the sets $A_t$ in the definition of the map  $\mu^\bullet X$ as subsets of $M_\cup X$. It was deduced from some general facts that the triple $\M_\cup^\bullet=(M_\cup,\eta,\mu^\bullet)$ is a monad \cite{Ny2}. For sake a completeness we give here a direct proof.

\begin{lemma}\label{cor} We have  $\mu^\bullet X(M_\cup(M_\cup X))\subset M_\cup X$ for each compactum $X$.
\end{lemma}

\begin{proof}  Consider any $\A\in M_\cup(M_\cup X)$ and $B$, $C\in \F(X)$. Since $B_t$ and $C_t$ are subsets of $M_\cup X$, we have $(C\cup B)_t=C_t\cup B_t$.  Then $\mu^\bullet X(\A)(B\cup C)=\max\{\A((C\cup B)_t)\cdot t\mid t\in(0,1]\}=\max\{\A(C_t\cup B_t)\cdot t\mid t\in(0,1]\}=\max\{\max\{\A(C_t)\cdot t\mid t\in(0,1]\},\max\{\A(B_t)\cdot t\mid t\in(0,1]\}=\max\{\mu^\bullet X(\A)(B),\mu^\bullet X(\A)(C)\}$.
\end{proof}

We will use the notation $\mu^\bullet X$ also for the restriction $\mu^\bullet X|_{M^2_\cup X}$.

\begin{theorem}\label{mon} The triple $\M_\cup^\bullet=(M_\cup,\eta,\mu^\bullet)$ is a monad.
\end{theorem}

\begin{proof}  It is easy to check that $\eta$ and $\mu^\bullet$ are well-defined natural transformations of corresponding functors. Let us check two monad properties.

Take any compactum $X$, $\nu\in M_\cup X$ and $A\in\F(X)$. Then we have $\mu^\bullet X\circ \eta M_\cup X(\nu)(A)=\max\{\eta \M_\cup X(\nu)(A_t)\cdot t\mid t\in(0,1]\}=\nu(A)$ and $\mu^\bullet X\circ M_\cup(\eta  X)(\nu)(A)=\max\{M_\cup(\eta  X)(\nu)(A_t)\cdot t\mid t\in(0,1]\}=\max\{\nu((\eta  X)^{-1}(A_t))\cdot t\mid t\in(0,1]\}=\max\{\nu(A)\cdot t\mid t\in(0,1]\}=\nu(A)$. We obtain the equality $\mu^\bullet\circ
M_\cup\eta=\mu^\bullet\circ\eta M_\cup=${\bf 1}$_{M_\cup}$.

Now, consider any $\gimel\in M_\cup^3(X)$ and $A\in\F(X)$. Put $a=\mu^\bullet X\circ M_\cup(\mu^\bullet X)(\gimel)(A)=\max\{\gimel((\mu^\bullet X)^{-1}(A_t))\cdot t\mid t\in(0,1]\}$ and $b=\mu^\bullet X\circ\mu^\bullet M_\cup X(\gimel)(\{a\})=$ \newline $=\max\{\mu^\bullet M_\cup X(\gimel)(A_t))\cdot t\mid t\in(0,1]\}=\max\{\max\{\gimel((A_t)_s)\cdot s\mid s\in(0,1]\}\cdot t\mid t\in(0,1]\}$.

There exists $t_0\in(0,1]$ such that $a=\gimel((\mu^\bullet X)^{-1}(A_{t_0}))\cdot t_0$. We have $(\mu^\bullet X)^{-1}(A_{t_0})=\{\A\in M_\cup^2(X)\mid\mu^\bullet X(\A)\ge t_0\}=\{\A\in M_\cup^2(X)\mid$ there exists $c\in(0,1]$ such that $\A(A_{c})\cdot c\ge t_0\}=\{\A\in M_\cup^2(X)\mid$ there exists $c\in(0,1]$ such that $\A(A_{c})\ge \frac{t_0}{c}\}$. Since $\gimel$ is a possibility capacity, there exists $\A_0\in M_\cup^2(X)$ and $c_0\in(0,1]$ such that $\A_0(A_{c_0})\ge \frac{t_0}{c_0}$ and $\gimel((\mu^\bullet X)^{-1}(A_{t_0}))=\gimel(\{\A_0\})$. But then we have $a\le\gimel((A_{c_0})_{\frac{t_0}{c_0}})\cdot t_0=\gimel((A_{c_0})_{\frac{t_0}{c_0}})\cdot\frac{t_0}{c_0}\cdot c_0\le b$.

On the other hand choose $p_0$, $z_0\in(0,1]$ such that $b=\gimel((A_{p_0})_{z_0})\cdot p_0\cdot z_0$. Since $\gimel$ is a possibility capacity,
there exists $\B_0\in (A_{p_0})_{z_0}$ such that $\gimel((A_{p_0})_{z_0})=\gimel(\{\B_0\})$. We have $\B_0(A_{p_0})\ge z_0$, hence $\mu^\bullet X(\B_0)(A)\ge z_0\cdot p_0$. Then we obtain $b=\gimel(\{\B_0\})\cdot p_0\cdot z_0\le\gimel((\mu^\bullet X)^{-1}(A_{p_0\cdot z_0}))\cdot p_0\cdot z_0\le a$.
\end{proof}

\section{Functional representation of the monad $\M_\cup^\bullet$}

A monad $\F=(F,\eta,\mu)$ is called an {\it IL-monad} if there exists
a map $\xi:FI\to I$ such that the pair $(I,\xi)$ is an
$\F$-algebra and
for each $X\in \Comp$ there exists a point-separating family of
$F$-algebras morphisms $\{f_\alpha:(FX,\mu X)\to(I,\xi)
\mid\alpha\in A\}$ \cite{Ra3}.

There was defined a monad $\V_I$ in \cite{Ra3}, which is universal in the class of IL-monads. By $V_IX$ we denote the power $I^{C(X,I)}$. For a map $\phi\in C(X,I)$
we denote by $\pi_\phi$ or $\pi(\phi)$ the
corresponding projection $\pi_\phi:V_IX\to I$. For each map $f:X\to Y$
we define the map $V_If:V_IX\to V_IY$ by the formula
$\pi_\phi\circ V_If=\pi_{\phi\circ f}$ for $\phi\in C(Y,I)$.
For a compactum $X$ we define components $hX$ and $mX$ of natural transformations by $\pi_\phi\circ hX=\phi$ and $\pi_\phi\circ m X=\pi(\pi_\phi)$ for all $\phi\in C(X,I)$). The triple $\V_I=(V_I,h,m)$ forms a monad in the
category $\Comp$ and for each monad $\F$ there exists a monad embedding $l:\F\to\V_I$ if and only if $\F$ is IL-monad \cite{Ra3}. Moreover, for a compactum $X$ the map $lX:FX\to V_IX$ is defined by the conditions $\pi_\phi\circ lX=\xi\circ F\phi$ for each $\psi\in C(X,I)$.

\begin{theorem} The monad $\M_\cup^\bullet$ is an IL-monad.
\end{theorem}

\begin{proof}  Define the map $\xi:M_\cup I\to I$ by the formula $\xi(\nu)=\max\{\nu([t,1]\cdot t\mid t\in(0,1]\}$. We can check that the pair $(I,\xi)$ is an $\M_\cup^\bullet$-algebra by the same but simpler arguments as in the proof of Theorem \ref{mon}.

Consider any compactum $X$ and two distinct capacities $\nu$, $\beta\in M_\cup X$. Then there exists $A\in\F(X)$ such that $\nu(A)\ne\beta(A)$. We can suppose that $\nu(A)<\beta(A)$. Since $\nu$ and $\beta$ are possibility capacities, there exist $a$, $b\in A$ such that $\nu(\{a\})=\nu(A)$ and $\beta(\{b\})=\beta(A)$.  Choose a point $t\in (\nu(A),\beta(A))$. Put $B=\{x\in X\mid \nu(\{x\})\ge t\}$. Since $\nu$ is a possibility capacity and $\nu(X)=1$, $B$ is not empty. Since $\nu$ is upper semicontinuous, $B$ is  closed. Evidently, $B\cap A=\emptyset$. Choose a function $\varphi\in C(X,I)$ such that $\varphi(B)\subset\{0\}$ and $\varphi(A)\subset\{1\}$. Then $\pi_\varphi\circ lX(\nu)=\xi\circ M_\cup\varphi(\nu)=\max\{M_\cup\varphi(\nu)([s,1]\cdot s\mid s\in(0,1]\}=\max\{\nu(\varphi^{-1}[s,1])\cdot s\mid s\in(0,1]\}\le t<\beta(A)\le\beta(\varphi^{-1}\{1\})\cdot 1\le\pi_\varphi\circ lX(\beta)$. It is easy to check that $\pi_\phi\circ lX=\xi\circ \M_\cup\phi:\M_\cup X\to I$ is an $\M_\cup^\bullet$-algebras morphism.
\end{proof}

Hence we obtain an monad embedding $l:\M_\cup^\bullet\to \V_I$ such that $\pi_\varphi\circ lX(\nu)=\max\{\nu(\varphi^{-1}[s,1])\cdot s\mid s\in(0,1]\}$ for each compactum $X$, $\nu\in M_\cup X$ and $\varphi\in C(X,I)$.

Let $X$ be any compactum. For any $c\in I$ we shall denote  by $c_X$ the
constant function on $X$ taking the value $c$. Following the notations of idempotent mathematics (see e.g., \cite{MS}) we use the
notation $\oplus$  in $I$ and $C(X,I)$ as an alternative for $\max$.
 We will use the notation $\nu(\varphi)=\pi_\varphi\circ lX(\nu)$ for $\nu\in V_IX$ and $\varphi\in C(X,I)$.

Consider the subset $SX\subset V_IX$ consisting of all functionals $\nu$ satisfying the following conditions

\begin{enumerate}
\item $\nu(1_X)=1$;
\item $\nu(\lambda\cdot\varphi)=\lambda\cdot\nu(\varphi)$ for each $\lambda\in I$ and $\varphi\in C(X,I)$;
\item $\nu(\psi\oplus\varphi)=\nu(\psi)\oplus\nu(\varphi)$ for each $\psi$, $\varphi\in C(X,I)$.
\end{enumerate}

Let us remark that properties 1 and 2 yield that $\nu(c_X)=c$ for each $\nu\in SX$ and $c\in I$.

\begin{theorem} $lX(M_\cup X)=SX$.
\end{theorem}

\begin{proof}  Consider any $\nu\in M_\cup X$. Put $\upsilon=lX(\nu)$. Then we have $\upsilon(1_X)=$ \newline $=\max\{\nu((1_X)^{-1}[s,1])\cdot s\mid s\in(0,1]\}=\max\{\nu(X)\cdot s\mid s\in(0,1]\}=1$.

Take any $c\in I$ and $\varphi\in C(X,I)$. For $c=0$ the Property 2 is trivial. For $c>0$  we have $\upsilon(c\varphi)=\max\{\nu((c\varphi)^{-1}[s,1])\cdot s\mid s\in(0,1]\}=\max\{\nu(\varphi^{-1}[\frac sc,1])\cdot \frac sc\mid s\in(0,1]\}\cdot c=c\cdot\upsilon(\varphi)$.

Consider any $\psi$ and $\varphi\in C(X,I)$. We have $\upsilon(\psi\oplus\varphi)=\max\{\nu((\psi\oplus\varphi)^{-1}[s,1])\cdot s\mid s\in(0,1]\}=\max\{\nu(\psi^{-1}[s,1]\cup\varphi^{-1}[s,1])\cdot s\mid s\in(0,1]\}=\max\{(\nu(\psi^{-1}[s,1])\oplus\nu(\varphi^{-1}[s,1]))\cdot s\mid s\in(0,1]\}=\upsilon(\psi)\oplus\upsilon(\varphi)$. We obtained  $lX(M_\cup X)\subset SX$.

Take any $\upsilon\in SX$. For $A\in\F(X)$ put $\Upsilon_A=\{\varphi\in C(X,I)\mid \varphi(a)=1$ for each $a\in A\}$. Define $\nu:\F(X)\to I$ as follows $\nu(A)=\inf\{\upsilon(\varphi)\mid \varphi\in \Upsilon_A\}$ if $A\ne\emptyset$ and $\nu(\emptyset)=0$. It is easy to see that $\nu$ satisfies Conditions 1 and 2 from the definition of capacity.

Let $\nu(A)<\eta$ for some $\eta\in I$ and $A\in\F(X)$. Then there exists $\varphi\in \Upsilon_A$ such that $\upsilon(\varphi)=\chi<\eta$. Choose $\varepsilon>0$ such that $(1+\varepsilon)\chi<\eta$. Put $\delta=\frac{1}{1+\varepsilon}$ and $\psi=\min\{\delta_X,\varphi\}$. Then $\upsilon(\psi)\le\upsilon(\varphi)=\chi$ and $\upsilon((1+\varepsilon)\psi)\le(1+\varepsilon)\chi<\eta$. Put $U=\varphi^{-1}(\delta,1]$. Evidently $U$ is an open set and $U\supset A$. But for each compact $K\subset U$ we have $ (1+\varepsilon)\psi\in \Upsilon_K$. Hence $\nu(K)<\eta$.

Finally take any $A$, $B\in\F(X)$. Evidently $\nu(A\cup B)\ge \nu(A)\oplus\nu(B)$. Suppose $\nu(A\cup B)>\nu(A)\oplus\nu(B)$. Then there exists $\varphi\in \Upsilon_A$ and $\psi\in \Upsilon_B$ such that $\nu(A\cup B)>\upsilon(\varphi)\oplus\upsilon(\psi)=\upsilon(\varphi\oplus\psi)$. But $\varphi\oplus\psi\in \Upsilon_{A\cup B}$ and we obtain a contradiction. Hence $\nu\in M_\cup X$.

Let us show that $lX(\nu)=\upsilon$. Take any $\varphi\in C(X,I)$. Denote $\varphi_t=\varphi^{-1}[t,1]$. Then $lX(\nu)(\varphi)=\max\{\inf\{\upsilon(\chi)\mid \chi\in \Upsilon_{\varphi_t}\}\cdot t\mid t\in(0,1]\}=\max\{\inf\{\upsilon(t\chi)\mid \chi\in \Upsilon_{\varphi_t}\}\mid t\in(0,1]\}$. For each $t\in(0,1]$ put $\chi_t=\min\{\frac{1}{t}\varphi,1_X\}\in \Upsilon_{\varphi_t}$. We have $t\chi\le\varphi$, hence  $\upsilon(t\chi)\le\upsilon(\varphi)$. Then we have $\inf\{\upsilon(t\chi)\mid \chi\in \Upsilon_{\varphi_t}\}\le\upsilon(\varphi)$ for each $t\in(0,1]$, hence $lX(\nu)(\varphi)\le\upsilon(\varphi)$.

Suppose $lX(\nu)(\varphi)<\upsilon(\varphi)$. Choose any $a\in(lX(\nu)(\varphi),\upsilon(\varphi))$. Then for each $t\in(0,1]$ there exists $\chi_t\in\Upsilon_{\varphi_t}$ such that $\upsilon(t\chi_t)<a$. Choose $\varepsilon>0$ such that $(1+\varepsilon)a<\upsilon(\varphi)$. Put $\delta=\frac{1}{1+\varepsilon}$. Choose $n\in\N$ such that $\delta^n<\upsilon(\varphi)$. Put $\psi_{n+1}=\delta^n_X$ and $\psi_{i}=\delta^{i-1}\chi_{\delta^i}$ for $i\in\{1,\dots,n\}$. We have $\upsilon(\psi_{i})<\upsilon(\varphi)$ for each $i\in\{1,\dots,n+1\}$.  Put $\psi=\oplus_{i=1}^{n+1}\psi_{i}$. Then  $\upsilon(\psi)=\oplus_{i=1}^{n+1}\upsilon(\psi_{i})<\upsilon(\varphi)$. On the other hand $\varphi\le\psi$ and we obtain a contradiction.
\end{proof}

Hence we obtain, in fact, that the monad $\M_\cup^\bullet$ is isomorphic to a submonad of $\V_I$ with functorial part acting on compactum $X$ as $SX$. Let us remark that this monad is one of monads generated by t-norms considered by Zarichnyi  \cite{Za}. Thus the following question seems to be natural: can we generalize the results of this paper to any continuous t-norms?

\end{document}